\documentclass[12 pt]{article}

\usepackage[utf8]{inputenc}
\usepackage{amsmath}
\usepackage{amsfonts}
\usepackage{amssymb}
\usepackage{mathtools}

\usepackage{cite}

\usepackage{graphicx}
\usepackage{caption}
\usepackage{listings}
\usepackage{verbatim}
\usepackage{wrapfig}
\usepackage{array}
\usepackage{subfigure}

\usepackage{tabularx}

\usepackage{caption}

\usepackage{amsthm}

\theoremstyle{lema}

\theoremstyle{proposition}

\theoremstyle{theorem}
\newtheorem{theorem}{Theorem}[section]

\theoremstyle{theorem}
\newtheorem{remark}{Remark}[section]

\theoremstyle{corollary}

\theoremstyle{definition}
\newtheorem{definition}{Definition}[section]

\theoremstyle{example}
\newtheorem{example}{Example}[section]

\def\n{\mathbb N}

\providecommand{\keywords}[1]
{
	\small	
	\textbf{\textit{Keywords---}} #1
}

\providecommand{\msc}[1]
{
	\small	
	\textbf{\textit{Mathematics Subject Classification---}} #1
}

\title{On fixed point results in metric spaces for large triangle-perimeter contractions}
\author{Ovidiu Popescu 
	\\ \small{Faculty of Mathematics and Computer Science,}\\ \small{Transilvania University of Bra\c sov, Bulevardul Eroilor 29, Bra\c sov}\\
	\small{email: {ovidiu.popescu@unitbv.ro}}}
\date{}

\begin{document}
	
	\maketitle

	\begin{abstract}
		In this paper we introduce a corrected extension of Burton's theory of large contractions in the context of triangle-perimeter contractions introduced by Petrov. Combining these two lines of research, we prove a fixed point result for large triangle-perimeter contractions with an auxiliary assumption, which is of utmost importance. Firstly, we give a counterexample to the main result related to large triangle-perimeter contractions that currently exists in literature. Then, we prove that given an additional condition, a fixed point result for large triangle-perimeter contractions holds. Lastly, we illustrate with an example that this new framework is strictly broader than Burton's and Petrov's theory.
	\end{abstract}
	
	
	\msc{47H10, 47H09}

	\section{Introduction}
	
	Starting with Banach's contraction principle in \cite{Banach}, the theory of fixed points has known an increasing development. Over the time, numerous mathematicians have generalized and extended Banach's theorem. Among them, we mention the contribution of Burton \cite{Burton2}, who introduced the concept of large contractions, where, compared to Banach's theorem, contractivity is not uniform on the entire space. For the newly introduced large contractions, it was proved that they still possess a unique fixed point if there exists a bounded orbit. This direction of research has been further developed by Dehici et al. in \cite{Dehici3}, where they studied large Kannan-type contractions combining Kannan's mappings \cite{Kannan} with Burton's large contractions. For these contractions, they showed that the orbit-boundness assumptions can be weakened, or even removed on bounded or compact metric spaces. Moreover, Mesmouli et al.  introduced in \cite{Mesmouli4} large Chatterjea-type contractions, combining large contractions with Chatterjea's mappings \cite{Chatterjea}. For these mappings, they presented some fixed-point results.
	
	Recently, Petrov \cite{Petrov} introduced triangle-perimeter contractions, which are defined through the perimeter of triangles formed by three iterates of a mapping.
	
	Based on Petrov's geometric method and Burton's nonuniform contractions, very recently Mesmouli et al. \cite{Mesmouli}, introduced the notion of large triangle-perimeter contraction. However, the fixed point result given does not hold in the context presented in \cite{Mesmouli}.
	
	The aim of this paper is to give a correct extension of Burton's theory  from \cite{Burton2} in the context introduced by Petrov in \cite{Petrov}, by proving a fixed point result for large-perimeter contractions with an auxiliary assumption, which is of utmost importance. Firstly, we give a counterexample to the main result in \cite{Mesmouli} (specifically, Theorem 3). Then, we prove that given an additional condition, a fixed point result for large triangle-perimeter contractions holds. Lastly, we prove that this new framework is strictly broader than Burton's and Petrov's theory.

	\section{Preliminaries}
	
	
	Let us recall some results related to large contractions and triangle-perimeter contractions.
	
	\begin{definition}[Burton \cite{Burton2}]
		Let \((X,d)\) be a complete metric space with \(|X| \geq 3\). A mapping \(T : X \to X \) is a large contraction if 
		\begin{itemize}
			\item[1)] \(d(Tx,Ty) < d(x,y)\), for all \(x \neq y\)
			\item[2)] for every \(\varepsilon > 0\), there exists a constant \(\delta(\varepsilon) \in (0,1)\) such that 
			\begin{equation*}
			x,y \in X, \; d(x,y) \geq \varepsilon \quad \text{implies} \quad d(Tx,Ty) \leq\delta(\varepsilon) d(x,y).
			\label{*}
			\end{equation*}
		\end{itemize} 
	\end{definition}

	\begin{theorem}[Burton \cite{Burton2}]
		Let $(X,d)$ be a complete metric space with \(|X| \geq 3\), and $T:X\to X$ be a large contraction. If there exists an $x_0 \in X$ and $L > 0$ such that 
		\[ d(x_0,T^nx_0) \leq L, \quad \forall n \geq 1,\]
		 then $T$ has a unique fixed point in \(X\).
	\end{theorem} 

	\begin{remark}
		In fact, condition 2) of large contractions implies condition 1), since if \(x\neq y\), then \(d(x,y) > 0\). Taking \(\varepsilon = d(x,y)\), there exists \(\delta(\varepsilon) \in (0,1)\) such that \(d(Tx,Ty) \leq \delta(\varepsilon) d(x,y) < d(x,y)\).
	\end{remark}

	\begin{definition}[Petrov \cite{Petrov}]
		Let $(X,d)$ be a metric space with  $|X|\geq 3$. A mapping $T \colon X\to X$ is called triangle-perimeter contraction if there exists a constant $\alpha \in [0,1)$ such that 
		\[P(Tx,Ty,Tz) \leq \alpha P(x,y,z),\]
		 for all three pairwise distinct points $x,y,z \in X$, where
		$$P(a,b,c) := d(a,b)+d(b,c)+d(a,c),$$
		represents the perimeter of the triangle with vertices \(a,b,c\).
	\end{definition}

	\begin{theorem}[Petrov \cite{Petrov}]
		Let $(X,d)$ be a complete metric space with  $|X|\geq 3$. If $T \colon X\to X$ is a triangle-perimeter contraction and has no periodic points of prime period $2$ (i.e. \(T^2x\neq x\), for every \(x \in X\)), then $T$ has a fixed point. The number of fixed points is at most two.
	\end{theorem}

	Very recently, the notion of large triangle-perimeter contraction was introduced combining the two aforementioned directions of research.

	\begin{definition}[Mesmouli \cite{Mesmouli}]
		Let $(X,d)$ be a metric space with  $|X|\geq 3$. A mapping $T \colon X\to X$ is a large triangle-perimeter contraction if 
		\begin{itemize}
			\item[(i)] for all pairwise distinct points $x,y,z \in X$, \(P(Tx,Ty,Tz) \leq \alpha P(x,y,z)\)
			\item[(ii)] for every \(\varepsilon > 0\), there exists a constant \(\delta(\varepsilon) \in (0,1)\) such that if
			\[\max\{d(x,y), d(y,z), d(x,z)\} \geq \varepsilon, \]
			then 
			\[P(Tx,Ty,Tz) \leq \delta(\varepsilon)  P(x,y,z). \]
		\end{itemize} 
	\label{largeP}
	\end{definition}

	\begin{remark}
		Condition (ii) implies condition (i), since if $x,y,z$ are pairwise distinct, taking \(\varepsilon = \max\{d(x,y), d(y,z), d(x,z)\} > 0\), there exists \(\delta(\varepsilon) \in (0,1)\) such that \(P(Tx,Ty,Tz) \leq \delta(\varepsilon)  P(x,y,z) < P(x,y,z)\).
	\end{remark}

	In \cite{Mesmouli}, a fixed point result for large triangle-perimeter contraction was given. However, the result does not hold without an additional assumption.

	\begin{theorem}[Mesmouli \cite{Mesmouli}]
		Let $(X,d)$ be a complete metric space with  $|X|\geq 3$, and let $T \colon X\to X$ be a large triangle-perimeter contraction in the sense of Definition \ref{largeP}. Assume that there exists an $x_0 \in X$ and $L > 0$ such that 
		\[ d(x_0,T^nx_0) \leq L, \quad \forall n \geq 1,\]
		then $T$ has a unique fixed point in \(X\).
	\end{theorem}

	\begin{example}[Mesmouli \cite{Mesmouli}]
		Let \(X=[0,1)\) with the metric \(d(x,y)=|x-y|\) and define \(Tx= \dfrac x{1+x}\). Then \(T\) is not a uniform triangle-perimeter contraction in the sense of Petrov, but it is a large triangle-perimeter contraction in the sense of Definition \ref{largeP}.
		\label{Ex2.1}
	\end{example}
	
	\begin{remark}
		However, the mapping \(T\) in Example \ref{Ex2.1} is a large contraction, since for every pair \((x,y)\) with \(d(x,y) \geq \varepsilon\), we have
		\[|Tx-Ty| = \dfrac{|x-y|}{(1+x)(1+y)} \leq \dfrac{|x-y|}{1+\varepsilon}. \]
		Hence, \(\delta(\varepsilon) = \dfrac{1}{1+\varepsilon} < 1\) verifies condition (ii) of large contractions.
	\end{remark}

	\begin{example}[Mesmouli \cite{Mesmouli}]
		Let \(X=\n_0=\{0,1,2, \dots\}\) with the metric \(d(m,n)=|m-n|\), and define \(T(n)= \left[\dfrac{n}{2}\right]\). Then \(T\) fails Petrov's uniform condition, but satisfies large triangle-perimeter contraction's condition.
		\label{Ex2.2}
	\end{example}

	\begin{remark}
		The mapping \(T\) in Example \ref{Ex2.2} is not a large contraction, as 
		\[d(T2,T1) = |T2-T1| = 1 = |2-1| = d(2,1).\]
		However, \(T\) satisfies Petrov's uniform condition.
		If \(m,n,p \in \n_0\) such that \(m>n>p\), we have 
		\[P(Tm,Tn,Tp) = 2\left(\left[\dfrac{m}{2}\right] - \left[\dfrac{p}{2}\right]\right),\]
		and \(P(m,n,p) = 2(m-p)\).
		For \(m=2k\), \(p=2l\), \(k,l \in \n_0\), \(k >l\), we obtain \(P(Tm,Tn,Tp) = 2(k-l)\) and \(P(m,n,p) = 4(k-l),\) so \(P(Tm,Tn,Tp) \leq \dfrac12 P(m,n,p)\).
		For \(m=2k+1\), \(p=2l\), \(k,l \in \n_0\), \(k >l\), we obtain \(P(Tm,Tn,Tp) = 2(k-l)\) and \(P(m,n,p) = 2(2k+1-2l),\) so \(P(Tm,Tn,Tp) \leq \dfrac12 P(m,n,p)\).
		For \(m=2k+1\), \(p=2l+1\), \(k,l \in \n_0\), \(k >l\), we obtain \(P(Tm,Tn,Tp) = 2(k-l)\) and \(P(m,n,p) = 4(k-l),\) so \(P(Tm,Tn,Tp) \leq \dfrac12 P(m,n,p)\).
		For \(m=2k\), \(p=2l+1\), \(k,l \in \n_0\), \(k >l\), we obtain \(P(Tm,Tn,Tp) = 2(k-l)\) and \(P(m,n,p) = 2(2k-2l-1),\) so \(P(Tm,Tn,Tp) \leq \dfrac34 P(m,n,p)\).
		
		Thus, \(T\) satisfies Petrov's uniform condition.
	\end{remark}

	\section{Main result}
	
	Let us precede our main result by an example which illustrates the necessity of imposing an additional condition to the fixed point result. 
	
	\begin{example}
		Let \(X=\{0,1,2\}\) and \(d(x,y)=|x-y|\). If $T \colon X\to X$, defined by \(T0=1\), \(T1=0\), \(T2=1\), we have 
		\[P(T0,T1,T2) = P(1,0,1) = 2 \quad \text{and} \quad P(0,1,2)=4.\]
		Thus, \(T\) is a large triangle-perimeter contraction (in fact, \(T\) satisfies Petrov's uniform condition), but \(T\) has no fixed points.
	\end{example}

	The previous example underlines the importance of imposing an additional condition in order to obtain a fixed point, specifically, requiring that the mapping does not possess periodic points of prime period 2.
	
	Let us now state and prove our main result:

	\begin{theorem}
		Let $(X,d)$ be a complete metric space with  $|X|\geq 3$ and let $T \colon X\to X$ be a large triangle-perimeter contraction in the sense of Definition \ref{largeP}. Suppose \(T\) has no periodic points of prime period 2 and there exists $x_0 \in X$ and \(L>0\) such that \(d(x_0, T^nx_0)\leq L \), for all \(n \geq 1\). Then, \(T\) has a fixed point in \(X\). The number of fixed points is at most two.
	\end{theorem}

	\begin{proof}
		Let \(x_0 \in X\) such that \(d(x_0, T^nx_0)\leq L \), for all \(n \geq 1\). Consider the Picard iteration \((x_n)_{n\geq 0}\) defined by \(x_{n+1}:=Tx_n=T^{n+1}x_0\) and let for each \(n \geq 0\):
	\[P_n:= P(x_n, x_{n+1}, x_{n+2}) = d(x_n,x_{n+1})+d(x_{n+1},x_{n+2})+d(x_n,x_{n+2}).\]
		If there exists \(n\) such that $x_{n+1}=Tx_n$, then \(x_n\) is a fixed point. Otherwise, we have \(x_n\neq Tx_n\), for all \(n \geq 0\). Since $T$ has no periodic points of prime period \(2\), \(x_{n+2}=T^2x_n \neq x_n\), for each $n \geq 0$. Therefore, from condition (i) of Definition \ref{largeP}, we get 
	\[P_{n+1} = P(Tx_n, Tx_{n+1}, Tx_{n+2}) < P(x_n, x_{n+1}, x_{n+2}) = P_n,\]
		for every $n \geq 0$.
		
		Hence, \((P_n)\) is strictly decreasing and bounded below by \(0\).
		
		Suppose there exists positive integers \(m>n\), such that \(x_m=x_n\). Then, we have
		\[ x_{m+1} = Tx_m = Tx_n = x_{n+1}, \quad x_{m+2} = T^2x_m = T^2x_n = x_{n+2}. \]
		
		Thus, 
		\[ P_m = P(x_m, x_{m+1}, x_{m+2}) = P(x_n, x_{n+1}, x_{n+2}) = P_n, \]
		which contradicts the strict monotonicity of \((P_n)\). Therefore, if no fixed point occurs along the orbit, all \(x_n\) are distinct.
		
		Now, we prove that \((x_n)\) is a Cauchy sequence.
		
		Assuming that \((x_n)\) is not Cauchy, there exists \(\varepsilon_0 > 0\), such that, for every positive integer \(N\) we can find \(m > n > N\), with \(d(x_m,x_n) \geq \varepsilon_0\). Hence, there exist the sequences of indices \(m_k\), \(n_k\), with \(m_k>n_k \to \infty\), such that
		\[d(x_{m_k}, x_{n_k})\geq \varepsilon_0, \quad \text{for each } k \geq 0. \]
		
		By condition (i) of Definition \ref{largeP}, we get 
		\[P(x_{m_k+2}, x_{m_k+1}, x_{n_k+1}) < P(x_{m_k+1}, x_{m_k}, x_{n_k}) < P(x_{m_k}, x_{m_k-1}, x_{n_k-1}) < \dots \]
		\[\quad \qquad \qquad \qquad < P(x_{m_k-n_k+1}, x_{m_k-n_k}, x_{0}).\]
		
		Since \(d(x_{m_k}, x_{n_k}) \leq P(x_{m_k+1}, x_{m_k}, x_{n_k}),\) we get
		\[\varepsilon_0 \leq P(x_{m_k+1}, x_{m_k}, x_{n_k}) \leq P(x_{m_k}, x_{m_k-1}, x_{n_k-1}) \leq \dots \leq P(x_{m_k-n_k+1}, x_{m_k-n_k}, x_{0}).\]
		
		Since 
		\[\varepsilon_0 \leq P(a,b,c) \quad \text{implies} \quad \dfrac{\varepsilon_0}{3}\leq \max \{d(a,b), d(b,c), d(a,c)\}, \]
		we have \[ \dfrac{\varepsilon_0}{3}\leq \max \{d(x_{m_k+1-i},x_{m_k-i}), d(x_{m_k-i},x_{n_k-i}), d(x_{m_k+1-i},x_{n_k-i})\}, \]
		for each \(k \geq 0\) and \(i \in \{0,1,\dots, n_k\}\). Hence, by condition (ii) of Definition \ref{largeP}, we get
		\[ P(x_{m_k+1}, x_{m_k}, x_{n_k}) \leq \delta\left(\dfrac{\varepsilon_0}{3}\right) P(x_{m_k}, x_{m_k-1}, x_{n_k-1}) \leq \left[\delta\left(\dfrac{\varepsilon_0}{3}\right)\right]^2 P(x_{m_k-1}, x_{m_k-2}, x_{n_k-2})  \]
		\[\quad \qquad \qquad \qquad \leq \dots \leq \left[\delta\left(\dfrac{\varepsilon_0}{3}\right)\right]^{n_k}  P(x_{m_k-n_k+1}, x_{m_k-n_k}, x_{0}).\]
		
		However,
		
		\begin{align*} P(x_{m_k-n_k+1}, x_{m_k-n_k}, x_{0}) &= d(x_{m_k-n_k+1},x_{m_k-n_k})+ d(x_{m_k-n_k+1},x_0)+d(x_{m_k-n_k},x_0)\\ &\leq 2 d(x_{m_k-n_k+1},x_0) + 2 d(x_{m_k-n_k},x_0)\leq 4L, \end{align*}
		for each $k \geq 0$. Therefore, we obtain
		\[P(x_{m_k+1}, x_{m_k}, x_{n_k}) \leq 4L \cdot \left[\delta\left(\dfrac{\varepsilon_0}{3}\right)\right]^{n_k}.  \]
		Hence, we obtain,
		\[\varepsilon_0 \leq d(x_{m_k}, x_{n_k}) \leq P(x_{m_k+1}, x_{m_k}, x_{n_k}) \leq 4L \cdot \left[\delta\left(\dfrac{\varepsilon_0}{3}\right)\right]^{n_k} \to 0, \]
		as $k \to \infty$, which is a contradiction. Then, we conclude that \((x_n)\) is a Cauchy sequence.
		
		Since $(X,d)$ is complete and \((x_n)\) is Cauchy, there exists \(x^*\in X\) such that \(x_n\to x^*\) as \(n \to \infty\).
		
		If \(x_n \neq x^*\), for every \(n \geq 0,\) then, by condition (i) of Definition \ref{largeP}, we obtain 
		\[ P(x_{n+2}, x_{n+1}, Tx^*) < P(x_{n+1}, x_{n}, x^*) \to 0,\]
		as \(n \to \infty\), so \[d(x^*, Tx^*) = \lim\limits_{n\to \infty} d(x_{n+1}, Tx^*) \leq \lim\limits_{n\to \infty} P(x_{n+2}, x_{n+1}, Tx^*) = 0,\] i.e., \(x^*\) is a fixed point of \(T\). Otherwise, there exists \(n_0 \geq 0\) such that \(x^*=x_{n_0}\). Since all \(x_n\) are distinct, we have \(x_n \neq x^*\), for every \(n \geq n_0\) and by condition (i) of Definition \ref{largeP}, we get
		\[ P(x_{n+2}, x_{n+1}, Tx^*) < P(x_{n+1}, x_{n}, x^*) \to 0,\]
		as \(n \to \infty\), so \(d(x^*, Tx^*) = 0\), i.e., \(x^*\) is a fixed point of \(T\).
		
		Suppose \(x^*, y^*, z^* \in X\) are three distinct fixed points of \(T\). Then, by condition (i) of Definition \ref{largeP}, we get:
		\[P(x^*, y^*, z^*) = P(Tx^*, Ty^*, Tz^*) < P(x^*, y^*, z^*),\]
		which is a contradiction.
		
		Therefore, \(T\) has at most two fixed points.
	\end{proof}

	\begin{remark}
		If there exists a Picard iteration \((x_n)\) such that \(x_n \to x^*\), with \(x_n \neq x^*\), for every \(n\), then \(T\) has a unique fixed point.
		
		Indeed, if \(y^*\) is another fixed point of \(T\), for \(n\) sufficiently large, we get \(x_n, x^*, y^*\) pairwise distinct, so 
		\[P(x_{n+1}, Tx^*, Ty^*) \leq \delta(d(x^*,y^*))P(x_n, x^*, y^*),\]
		by where, taking the limit as \(n \to \infty\), we obtain		
		\[P(x^*, x^*, y^*) \leq \delta(d(x^*,y^*))P(x^*, x^*, y^*) < P(x^*,x^*, y^*),\]
		which is a contradiction.
	\end{remark}

	\begin{example}
		Let \(X = [0,1] \cup \{ 4n, 4n+1\; : \;n \in \n \}\) and \(d(x,y)=|x-y|\). Let $T \colon X\to X$ be defined by \(Tx = \dfrac{x}{x+1}\), if \(x \in [0,1]\) and \(T(4n) = 0\), \(T(4n+1) = 1-\dfrac1n\). 
		
		Obviously, $(X,d)$ is a complete metric space. We show that \(T\) is not a large contraction and it is not a uniform triangle-perimeter contraction in the sense of Petrov, but it is a large contraction in the sense of Definition \ref{largeP}.
		
		\begin{enumerate}
			\item[1.] Suppose \(T\) is a large contraction. Then, for \(\varepsilon = 1\), there exists \(\delta(\varepsilon) \in (0,1)\) such that
			\[d(x,y) \geq 1 \quad \implies \quad d(Tx,Ty)\leq \delta(\varepsilon). \]
			Taking \(x=4n+1\) and \(y=4n\), we have \(d(4n,4n+1) = 1 \geq \varepsilon\), so \[d(T(4n+1),T(4n)) \leq \delta(\varepsilon). \] This implies \(1-\dfrac1n \leq \delta(\varepsilon), \) which is not true if \(n \to \infty\).
			
			\item[2.] Taking \(x=\varepsilon\), \(y=2\varepsilon\), \(z=3\varepsilon\), with \(\varepsilon \in [0,\frac13]\), we have
			\[ \dfrac{P(Tx, Ty, Tz)}{P(x, y, z)} = \dfrac{1}{(\varepsilon+1)(3\varepsilon+1)} \to 1 \quad \text{as}  \quad \varepsilon \to 0,\]
			so \(T\) is not a uniform triangle-perimeter contraction in the sense of Petrov.
			
			\item[3.] For \(0\leq x < y < z \leq 1\), with \(d(z,x) \geq \varepsilon\), we have \(z \geq \varepsilon\), so 
			\[ \dfrac{P(Tx, Ty, Tz)}{P(x, y, z)} = \dfrac{1}{(x+1)(z+1)} \leq \dfrac1{1+\varepsilon} = \delta(\varepsilon) < 1.\]
			
			For \(0\leq x < y \leq 1\) and \(z = 4n\), we have \(d(Tx,Ty) < 1\), \(d(Ty,Tz) < 1\), \(d(Tx,Tz) < 1\), so 
			\[ \dfrac{P(Tx, Ty, Tz)}{P(x, y, z)} = \dfrac{3}{8n-2x} \leq \dfrac12,\]
			for every \(n \geq 1\) and \(x \in [0,1]\).
			
			If \(0\leq x < y \leq 1\) and \(z = 4n+1\), we have \(d(Tx,Ty) < 1\), \(d(Ty,Tz) < 1\), \(d(Tx,Tz) < 1\), so 
			\[ \dfrac{P(Tx, Ty, Tz)}{P(x, y, z)} = \dfrac{3}{8n+2-2x} \leq \dfrac12,\]
			for every \(n \geq 1\) and \(x \in [0,1]\).
			
			If \(0\leq x \leq 1\) and \(y,z \in \{ 4n, 4n+1\; : \;n \in \n \} \), we have \(d(Tx,Ty) < 1\), \(d(Ty,Tz) < 1\), \(d(Tx,Tz) < 1\) and \(P(x, y, z) \geq 3+4+1 = 8\), so
			\[ \dfrac{P(Tx, Ty, Tz)}{P(x, y, z)} = \dfrac{3}{8} < \dfrac12.\]
			
			For \(x,y,z \in \{ 4n, 4n+1\; : \;n \in \n \} \), we have \(d(Tx,Ty) < 1\), \(d(Ty,Tz) < 1\), \(d(Tx,Tz) < 1\) and \(P(x, y, z) \geq 1+3+4 = 8\), so
			\[ \dfrac{P(Tx, Ty, Tz)}{P(x, y, z)} = \dfrac{3}{8} < \dfrac12.\]
			
			Hence, taking \(\delta(\varepsilon) = \dfrac{1}{1+\varepsilon}\) when \(\varepsilon \in [0,1]\) and \(\delta(\varepsilon) = \dfrac12\) when \(\varepsilon >1\), we have
			\(P(Tx,Ty,Tz)  \leq \delta(\varepsilon) P(x,y,z).\)
			
			Thus, \(T\) is a large triangle-perimeter contraction in the sense of Definition \ref{largeP}.
			
		\end{enumerate}
	\end{example}
	
	\section{Conclusions}

	In this paper, we gave a counterexample to the main result stated in Theorem 3 from \cite{Mesmouli}, and proved that under an auxiliary condition, the results hold. Also, we showed that examples from \cite{Mesmouli} are not suitable and presented a new example which shows that this new framework is strictly broader than Burton's and Petrov's theory.

	\medskip
	\vspace{1.2ex}
	
\end{document}